\documentclass[11pt,a4paper]{amsart}
\usepackage[margin=3cm]{geometry}
\usepackage{color}
\usepackage{graphicx,caption,subcaption}                         

\usepackage[pdftex,hyperfootnotes=false,colorlinks=true,linkcolor=blue,
citecolor=purple,filecolor=magenta,urlcolor=cyan]{hyperref}
\usepackage[dvipsnames]{xcolor} 

\usepackage{amsmath,amssymb,amscd,amsthm,amsfonts,anyfontsize,dsfont,%
enumerate,fix-cm,layout,lipsum,lpic,mathrsfs,mdwlist,pigpen,stmaryrd,%
tensor,thmtools,tikz,txfonts,xspace}
\usepackage[all]{xy}
\usepackage[oldstyle]{libertine}%
\usepackage[T1]{fontenc}
\usepackage[utf8]{inputenc}
\usepackage{csquotes}
\makeatletter
\renewcommand*\libertine@figurestyle{LF}
\makeatother
\usepackage[libertine,libaltvw,liby]{newtxmath}
\makeatletter
\renewcommand*\libertine@figurestyle{OsF}
\makeatother
\usepackage{fancyhdr}
\usepackage{tikz-cd}
\usepackage[noabbrev]{cleveref}
\expandafter\def\csname ver@etex.sty\endcsname{3000/12/31}

\usepackage{autonum}
\newtheorem{theorem}{Theorem}[section]
\newtheorem{lemma}[theorem]{Lemma}
\newtheorem{proposition}[theorem]{Proposition}

\theoremstyle{definition}
\newtheorem{definition}[theorem]{Definition}
\newtheorem{remark}[theorem]{Remark}

\usetikzlibrary{positioning}

\usepackage{hyperref} 
\definecolor{webbrown}{rgb}{0.65, 0.16, 0.16} 
\hypersetup{
colorlinks=true, linktocpage=true, pdfstartpage=1, pdfstartview=FitV,breaklinks=true, pdfpagemode=UseNone, pageanchor=true, pdfpagemode=UseOutlines,plainpages=false, bookmarksnumbered, bookmarksopen=true, bookmarksopenlevel=1,hypertexnames=true, pdfhighlight=/O,urlcolor=webbrown, linkcolor=RoyalBlue, citecolor=ForestGreen}

\usepackage[textsize=footnotesize]{todonotes}	
\presetkeys%
		{todonotes}%
		{color=Apricot}{}%
\newcommand{\mb}[1]{\mathbb{#1}} 

\newcommand{\mc}[1]{\mathcal{#1}}

\newcommand{\N}{\mb{N}} 
\newcommand{\C}{\mb{C}} 
\newcommand{\Z}{\mb{Z}} 
\newcommand{\Q}{\mb{Q}}
\newcommand{\E}{\mathcal{E}}

\newcommand{\mZeta}[6]{\varsigma\left( \begin{smallmatrix}
 #1 & #2 & #3 \\
 #4 & #5 & #6
 \end{smallmatrix} 
 \right)}

\def\CP1{\mathbb{C}\mathrm{P}^1}

\DeclareMathOperator{\Aut}{Aut}

\newcommand{\id}{\mathord{\mathrm{id}}}

\begingroup\expandafter\expandafter\expandafter\endgroup
\expandafter\ifx\csname pdfsuppresswarningpagegroup\endcsname\relax
\else
\fi
{}

\begin{document}
\title[completed cycles Leaky Hurwitz numbers]{completed cycles Leaky Hurwitz numbers}

\author[D.~Accadia]{Davide Accadia}
\address{D.~A.: Dipartimento di Matematica, Informatica e Geoscienze, Università di Trieste, via Weiss 8, 34128 Trieste (IT)}
\email{davide.accadia@phd.units.it}
\author[M.~Karev]{Maxim Karev}
\address{M.~K: Guangdong-Technion Israel Institute of Technology, 241 Daxue road, Guangdong, 515063 Shantou (CN)}
\email{maksim.karev@gtiit.edu.cn}
\author[D.~Lewa\'nski]{Danilo Lewa\'nski}
\address{D.~L.: Dipartimento di Matematica, Informatica e Geoscienze, Università di Trieste, via Weiss 8, 34128 Trieste (IT)}
\email{danilo.lewanski@units.it}
\thanks{}

\begin{abstract}
We introduce $(r+1)$-completed cycles $k$-leaky Hurwitz numbers and prove piecewise polynomiality as well as establishing their chamber polynomiality structure and their wall crossing formulae. For $k=0$ the results recover previous results of Shadrin-Spitz-Zvonkine. The specialization for $r=1$ recovers Hurwitz numbers that are close to the ones studied by Cavalieri-Markwig-Ranganathan and Cavalieri-Markwig-Schmitt. The ramifications differ by a lower order \emph{torus correction} \textemdash \; not affecting the genus zero enumeration, nor the enumeration for leaky parameter values $k = \pm 1$ in all genera \textemdash \; which is however natural from the Fock space perspective. 
\end{abstract}

\maketitle
\tableofcontents

\section{Introduction}
\label{sec:intro}

Hurwitz numbers have been introduced by Adolf Hurwitz in \cite{Hur} as the enumeration of genus $g$, degree $d$ branched coverings of the Riemann sphere with given ramification profiles over a number of given fixed points. Over the last three decades, there have been many developments in the theory of Hurwitz numbers in different branches of mathematics and physics, including enumerative algebraic geometry, differential geometry, tropical geometry, combinatorics, representation theory,  integrable systems, and random matrix models. For a recent textbook on Hurwitz theory we refer to \cite{CM}.

Different types of Hurwitz numbers have proven to be of particular interest, and over type the scientific community populated a rich dictionary of adjectives in Hurwitz theory, which are more or less useful depending on the property studied or on the link with other branches of mathematics and physics. For instance, \emph{connected} (resp. \emph{possibly disconnected}) Hurwitz numbers refer to the topological connection (resp. possible disconnection) of the target ramified covering. Similarly, \emph{single} Hurwitz numbers (resp. \emph{double}) refer to the presence of one (resp. two) given ramification profiles over fixed points, whereas the other ramifications often fill out the remaining ramification locus with a given rule of a given type: \emph{simple} corresponds to filling out the ramification locus with simple ramification profiles, and is alternative to \emph{$(r+1)$-completed cycles} corresponding instead to filling with precise linear combinations of higher ramification index points related to the Gromov-Witten theory of smooth algebraic curves \cite{OP}. \emph{Tropical} Hurwitz numbers refers to their tropical geometric interpretation in terms of tropical covers, \emph{$b$-Hurwitz numbers} are related to not necessarily orientable covers provided with a way to measure the degree of non-orientability, \emph{strictly monotone} or \emph{weakly monotone} Hurwitz numbers refer to certain monotonicity condition on the sheets involved in the simple ramifications, Bousquet-M\'elou--Schaeffer numbers or Orlov--Scherbin numbers also can be cast in the framework of Hurwitz theory, as well as Grothendieck dessins d'enfant and many more \dots

It is a natural question to consider Hurwitz numbers as functions of the parts of the fixed ramification profiles (i.e. the ones coming from the single or double characterization). For example, it was observed in \cite{GJ} that single simple Hurwitz numbers in genus zero exhibit polynomiality in the entries of the only fixed ramification profile, up to a combinatorial prefactor, later generalized to arbitrary genus in \cite{ELSV}. Another feature often present in Hurwitz theory is the so-called piecewise polynomiality or chamber polynomiality, that is, certain types of double possibly disconnected Hurwitz numbers are piecewise polynomials simultaneously in the parts of the two fixed ramifications, where the difference between two polynomials in adjacent chambers separated by a wall is referred to as wall-crossing formula whenever it can be computed explicitly. We now recall the state of the art around the properties of Hurwitz numbers we are interested in for this work.

Double Tropical Hurwitz numbers have been studied in \cite{CJM10} and their wall crossing formulae established in \cite{CJM11}. In \cite{CJMR} the tropicalization of the celebrated Gromov-Witten/Hurwitz (GW/H, for shortness) correspondence \cite{OP} was established. The main computational tool employed to complete the original correspondence is the semi-infinite wedge, found in Hurwitz theory by \cite{Ok00} and later exploited in \cite{John} to prove the piecewise polynomiality of double Hurwitz numbers. Chamber polynomiality for Hurwitz numbers in genus zero was already investigated in \cite{SSV08} and the geometry of double Hurwitz numbers was studied in \cite{GJV05}.

 The same algorithm used by Johnson was then employed in \cite{SSZ} for completed cycles Hurwitz numbers, in \cite{AKL} for monotone and strictly monotone Hurwitz numbers (similar arguments appear in tropical Jucys covers \cite{AL, AL2}). The same formalism was also employed to prove the quasi-polynomiality of Hurwitz number in \cite{DBLPS, KLS, KLPS} and to reprove the Harer-Zagier formula for the Euler characteristic of the smooth moduli spaces \cite{L1}, among other applications.

The study of leaky Hurwitz numbers was initiated in \cite{CMR} and further investigated in their generalization with descendants in \cite{CMS}. 

In \cite{CMR} $k$-leaky double Hurwitz numbers were defined from logarithmic intersection theory on
pluricanonical double ramification cycles, recast as count of tropical curves (i.e. as tropical Hurwitz numbers) and expressed in terms of the Fock space formalism. It is moreover shown how to recover traditional double
Hurwitz numbers in this setting. Their piecewise polynomiality is established, though no wall-crossing formula is derived. 

In \cite{CMS} a new class of invariants is defined adding descendants, again directly from intersection theory. These numbers are called $k$-leaky double Hurwitz descendants, which generalize and interpolate between the $k$-leaky double Hurwitz numbers defined in \cite{CMR} and descendant integrals of double ramification cycles. Their interpretation in terms of tropical covers is described, and the piecewise polynomiality and the
the wall-crossing formula in genus zero is proved, together with the proof of their non-negativity and a complete characterization of their vanishing, again in the genus zero sector.

The origin of the role that integrability plays in enumerative algebraic geometry of moduli spaces of stable curves can be traced back to Witten conjecture~\cite{wit-1}, shortly after proved by Kontsevich in \cite{kon}. The integrability of Hurwitz number was conjectured by Pandharipande \cite{P} and shortly after proved by the already cited \cite{Ok00}. More generally, the key concept behind this interaction is represented by the definition of Cohomological Field Theory introduced by Kontsevich and Manin \cite{KM}, later followed by the Givental-Teleman reconstruction theory (see e.g. \cite{Tel}).

Hurwitz numbers have been proved to be intrinsically related to the intersection theory of the moduli space of stable curves in \cite{ELSV}. After that, many types of ELSV-type formulae have been found around different kinds of Hurwitz problems, that is, different kinds of Hurwitz numbers defined by changing the conditions on the possible ramifications. ELSV formulae have been found for example for monotone Hurwitz number \cite{ALS}, for Hurwitz numbers with a full ramification \cite{DL}, for orbifold Hurwitz numbers \cite{JPT, LPSZ} for the Euler characteristic of the smooth moduli space of curves in \cite{GLN} for double Hurwitz numbers \cite{doubleTR}, later deformed in \cite{ELSVTR}, to approach hyperelliptic Hodge integrals \cite{L3} and many other examples (for an overview on the topic see e.g. \cite{L2}). Several of these investigations were possible thanks to the computational Sage package \cite{DSvZ21}.

\subsubsection*{Results}

We continue this investigation by defining the $(r+1)$-completed cycles generalization of $k$-leaky double Hurwitz numbers, by defining them in terms of the Fock space, which in turn defines their tropical interpretation along the motto \emph{bosonification is tropicalization}. For $r=1$, we recover invariants that are close to the $k$-leaky double Hurwitz numbers of \cite{CMR}, although a correction arises naturally from the Fock space point of view. This correction 
is not present for the leaky parameter values $k=1$ and $k=-1$. Summarising, given a Hurwitz problem the following are among the most interesting aspects, or theorems, of study:

\begin{enumerate}
\item Geometric interpretation in terms of ramified covers of surfaces (or variations)
\item Tropical interpretation
\item Group algebraic interpretation
\item Character theoretical interpretation
\item Fock space vacuum expectation value interpretation
\item Piecewise polynomiality for possibly disconnected (although connected in the chambers) double Hurwitz numbers
\item Wall crossing formulae for possibly disconnected (although connected in the chambers) double Hurwitz numbers
\item Cut-and-join equation (usually for possibly disconnected) single or double Hurwitz numbers
\item Topological recursion for connected (usually for single, then for double) Hurwitz numbers
\item ELSV-type formula for connected (usually for single, then for double) Hurwitz numbers
\end{enumerate}

Essentially, in this paper, we define $(r+1)$-completed cycles $k$-leaky double Hurwitz numbers by (5), which also defines (2), and then we address and establish points (6), (7), (8). Our methods generalize and rely on the algorithm introduced by Johnson in \cite{John} for simple double Hurwitz numbers and later extended in \cite{SSZ} to the case of $(r+1)$-completed cycles nonleaky ($k=0$) double Hurwitz numbers.

\subsection{Organisation of the paper}
The paper begins with a preliminary section to recall the Fock operators in the semi-infinite wedge formalism \ref{sec:fock}, followed by the description of the Hurwitz numbers in the exam as vacuum expectation values, section \ref{sec:Hurwitz VEVs}.
In section \ref{sec:piecewise} we establish the first result of the paper, that is, the piecewise polynomiality of the double leaky Hurwitz numbers with completed cycles, and we develop their wall crossing formula in section \ref{sec:wallcrossing}. Section \ref{sec:cutnjoin} is devoted to the cut-and-join operators for the Hurwitz generating series. Section \ref{sec:geometry} discusses certain tropical geometric aspects and compares leaky Hurwitz numbers existing in the literature, finally section \ref{sec:conclusions} proposes research questions for the future development of leaky Hurwitz numbers in relation to neighboring branches of mathematics and physics.

\subsection{Acknowledgements} 
D. A. and D. L. are supported by the University of Trieste. D. L. is supported by the INdAM group GNSAGA, and by the INFN within the project MMNLP (APINE). The authors thank Cavalieri and Schmitt for useful discussions. 

\vspace{1cm}
\section{Semi-infinite wedge formalism and VEVs computation}
\label{sec:fock}

We introduce the operators needed for the derivation of our results. For a self-contained introduction to the infinite wedge space formalism, we refer the reader to \cite{OP,John}, where most relevant objects are defined. 
\par
Let $V = \bigoplus_{i \in \Z + \frac 12} \C \underline{i}$ be an infinite-dimensional complex vector space with a basis labeled by half-integers, written as \( \underline{i}\) for $i \in \Z + \frac 12$. The semi-infinite wedge space \( \mc{V} \coloneq \bigwedge^{\frac{\infty}{2}} V \) is the space spanned by vectors
\begin{equation}
\underline{k_1} \wedge \underline{k_2} \wedge \underline{k_3} \wedge \cdots
\end{equation}
such that for large \( i\), \( k_i + i -\frac{1}{2} \) equals a constant, called the charge, imposing that \( \wedge \) is antisymmetric. The charge-zero sector 
\begin{equation}
\mathcal{V}_0 = \bigoplus_{n \in \mathbb{N} } \bigoplus_{\lambda\,  \vdash\, n} \C v_{\lambda}
\end{equation}
is the span of all of the semi-infinite wedge products \( v_{\lambda} = \underline{\lambda_1 - \frac{1}{2}} \wedge \underline{\lambda_2 - \frac{3}{2}} \wedge \cdots \) for some integer partition~$\lambda$. The space $\mathcal{V}_0$ has a natural inner product $(\cdot,\cdot )$ defined by declaring elements $v_{\lambda}$ to be orthonormal. 
The element corresponding to the empty partition $v_{\emptyset}$ is called the vacuum vector and denoted by $|0\rangle$. Similarly, we call its dual in \( \mc{V}_0^* \) the covacuum vector, and denote it $\langle0|$. If $\mathcal{P}$ is an operator acting on $\mathcal{V}_0$, we denote with $ \langle \mathcal{P}\rangle$ the evaluation $ \langle 0 |  \mathcal{P} |0\rangle$.

For $k$ half-integer, define the operator $\psi_k$ by $\psi_k : (\underline{i_1} \wedge \underline{i_2} \wedge \cdots) \ \mapsto \ (\underline{k} \wedge \underline{i_1} \wedge \underline{i_2} \wedge \cdots)$, and let $\psi_k^{\dagger}$ be its adjoint operator with respect to~$(\cdot,\cdot)$. The normally ordered products of $\psi$-operators
\begin{equation}
E_{i,j} \coloneqq \begin{cases}\psi_i \psi_j^{\dagger}, & \text{ if } j > 0 \\
-\psi_j^{\dagger} \psi_i & \text{ if } j < 0 \end{cases} 
\end{equation}
are well-defined operators on $\mathcal{V}_0$. We will use in the following definitions and computations the functions
\begin{equation}
\varsigma (z)=e^{z/2} - e^{-z/2} = 2 \sinh(z/2), \qquad \qquad 
\mathcal{S}(z) = \frac{\varsigma (z)}{z}.
\end{equation} 
For $n$ any integer, and $z$ a formal variable, define the operators
\begin{equation}\label{eq:defE}
\mathcal{E}_n(z) = \sum_{k \in \Z + \frac12} e^{z(k - \frac{n}{2})} E_{k-n,k} + \frac{\delta_{n,0}}{\varsigma(z)}, \qquad \qquad \alpha_n = \mathcal{E}_n(0) = \sum_{k \in \Z + \frac12} E_{k-n,k}.
\end{equation}
Their commutation formulae are known to be
\begin{equation}\label{eq:commE}
  \left[\mathcal{E}_a(z),\mathcal{E}_b(w)\right] =
\varsigma\left(\det  \left[
\begin{matrix}
  a & z \\
b & w
\end{matrix}\right]\right)
\,
\mathcal{E}_{a+b}(z+w), 
\qquad \qquad 
[\alpha_k, \alpha_l ] = k \delta_{k+l,0}.
\end{equation}
 We will also use the $\mathcal{E}$-operator without the correction in energy zero, i.e.
\begin{equation}
\tilde{\mathcal{E}}_0(z) = \sum_{k \in \Z + \frac12} e^{zk} E_{k,k} = \sum_{r=0}^{\infty} \mathcal{F}_r z^r, 
\qquad \qquad
 \mathcal{F}^{(r)} \coloneqq \sum_{k\in\Z+\frac12} \frac{k^r}{r!} E_{k,k} 
\end{equation}
The operator $E = \mathcal{F}^{(1)}$ is called the \emph{energy operator} and the operator $C = \mathcal{F}^{(0)}$ is called the \emph{charge operator}. If
$$
[\mathcal{O}, \mathcal{F}^{(1)}] = e\mathcal{O},
$$ for some $e \in \Z$, we say that the
\emph{energy} of an operator $\mathcal{O}$ is equal to  e. Similarly, if $[\mathcal{O}, \mathcal{F}^{(0)}] = c\mathcal{O},$ we say that the \emph{charge} of $\mathcal{O}$ is $c$. If $\mathcal{O},\mathcal{Q}$ are operators with energies $e_1,e_2$ and charges $c_1,c_2,$ respectively, then their commutator $[\mathcal{O},\mathcal{Q}]$ has energy $e_1+e_2$ and charge $c_1+c_2.$
The operators $\mathcal{E}_{n}(z)$ and consequently the operators $\mathcal{F}^{(r)}$ and $\alpha_n$ have zero charge. The operators $\mathcal{E}_{n}(z)$ have energy $n$, and so do the operators $\alpha_n$, whereas for all $r\ge 0$ the operators $\mathcal{F}^{(r)}$ have energy $0$.

The Boson-Fermion correspondence (see the foundational book \cite{Solitons}, but also \cite{Ellena} for all details worked out) implies that the operators $\mathcal{E}_k(z)$ can themselves be expressed in terms of the $\alpha_\ell$ operators as 
\begin{equation}\label{eq:bosonfermion}
\mathcal{E}_k(z) = \frac{1}{\varsigma(z)}\sum_{n=0}^\infty \frac{1}{n!} \sum_{\substack{k_1,\dots,k_n = 0 \\ k_1 + \dots + k_{n} = k} } \varsigma (k_1 z) \cdots \varsigma (k_n z)  : \!\frac{\alpha_{k_1}}{k_1} \cdots \frac{\alpha_{k_n}}{k_n} \!: \;\;+\;\;  \frac{\delta_{n,0}}{\varsigma(z)}
\end{equation}

\subsection{Computation of VEVs}\label{subsec:computation}
In the following, we describe several simple facts that follow from the definitions as well as a handy algorithm to compute vacuum expectation values. For more details see the foundational book \cite{Solitons} as well as the paper \cite{John} in which these considerations were applied for explicit computations of Hurwitz numbers after the main result in \cite{Ok00} was established. 

\begin{definition}
Define the vacuum expectation value (VEV for short) of a Fock space operator $\mathcal{O}$ as
\begin{equation}
\langle \mathcal{O} \rangle \coloneqq (v_{\emptyset}, \mathcal{O}.v_{\emptyset}),
\end{equation}
where $(\bullet, \bullet)$ is the inner product defined by declaring the elements $v_{\lambda}$ for all partitions $\lambda$ orthonormal. In other words, the VEV $\langle \mathcal{O} \rangle$ enumerates how many summands equal to $v_{\emptyset}$ the operator $\mathcal{O} $ produces when applied to $v_{\emptyset}$, and weights them with the constant that gets produced by the application of the operator.
\end{definition}

 We provide here useful observations concerning the general vanishing of VEVs as well as an algorithm to compute them. We assume that we deal with operators with defined energy, and use a convention in which the index of an operator refers to its energy (this convention also applies to the operators $\mathcal{E}_k(z)$).

\begin{itemize}
\item[\emph{i).}] If $\mathcal{O}_e$ has negative energy $e<0$, then we have the vanishing of any VEV with $\mathcal{O}_e$ on the very left:
\begin{equation}
e<0 \; \implies \; \left\langle \mathcal{O}_e \prod_{i=1}^N \mathcal{O}_{k_i} \right\rangle =0, \qquad k_i \in \mathbb{Z}
\end{equation}
The vanishing also holds for zero energy operators, unless they are the identity operator, in which case they can be absorbed by multiplication with the identity element.
\item[\emph{ii).}] Adjointly, if $\mathcal{O}_e$ has positive energy $e>0$, then we have the vanishing of any VEV with $\mathcal{O}_e$ on the very right:
\begin{equation}
e>0 \; \implies \; \left\langle  \prod_{i=1}^N \mathcal{O}_{k_i} \mathcal{O}_e \right\rangle =0, \qquad k_i \in \mathbb{Z}
\end{equation}
The vanishing also holds for zero energy operators, unless they are the identity operator, in which case they can be absorbed by multiplication with the identity element.
\item[\emph{iii).}] VEVs of non-zero total energy vanish:
\begin{equation}
\sum_{i=1}^N k_i \neq 0 \; \implies \; \left\langle  \prod_{i=1}^N \mathcal{O}_{k_i} \right\rangle =0.
\end{equation}
\item[\emph{iv).}] If the second leftmost operator has negative energy, then by fact $i).$ commuting the first two operators results in a single term:
\begin{equation}
e_2<0 \; \implies \; \left\langle \mathcal{O}_{e_1} \mathcal{O}_{e_2}  \prod_{i=1}^N \mathcal{O}_{k_i} \right\rangle = \left\langle \left[\mathcal{O}_{e_1}, \mathcal{O}_{e_2} \right] \prod_{i=1}^N \mathcal{O}_{k_i} \right\rangle, \qquad k_i \in \mathbb{Z}.
\end{equation}
The vanishing also holds for zero energy operators, unless they are the identity operator, in which case they can be absorbed by multiplication with the identity element.
\item[\emph{v).}] Adjointly, if the second right-most operator has positive energy, then by fact $ii).$ commuting the last two operators results in a single term:
\begin{equation}
e_1>0 \; \implies \; \left\langle \prod_{i=1}^N \mathcal{O}_{k_i} \mathcal{O}_{e_1} \mathcal{O}_{e_2} \right\rangle = \left\langle \prod_{i=1}^N \mathcal{O}_{k_i}  \left[\mathcal{O}_{e_1}, \mathcal{O}_{e_2} \right] \right\rangle, \qquad k_i \in \mathbb{Z}
\end{equation}
The vanishing also holds for zero energy operators, unless they are the identity operator, in which case they can be absorbed by multiplication with the identity element.
\item[\emph{vi).}] If a VEV of $N$ operators is non zero, after absorption of possible identity operators, by $i).$ and $ii).$ the left (right) most operator $\mathcal{O}_e$ has to have positive (negative) energy. Hence one can start with the left (right) most operator of the VEV, and commute it in a finite number of steps all the way passed the remaining $N-1$ operators to the right (left), producing many terms along the way. The produced terms are VEVs containing $N-2$ initial operators and exactly one commutator between two of them, with the only exception of the last term in which the operator $\mathcal{O}_e$ reached the last position on the right (left), which indeed vanishes by $i).$ (respectively, by $ii).$).
\begin{align}
\left\langle \mathcal{O}_{e_1} \mathcal{O}_{e_2} \cdots \mathcal{O}_{e_N} \right\rangle &= \left\langle \left[\mathcal{O}_{e_1}, \mathcal{O}_{e_2}\right] \mathcal{O}_{e_3} \cdots \mathcal{O}_{e_N}\right\rangle  + \left\langle \mathcal{O}_{e_2} \mathcal{O}_{e_1}  \cdots \mathcal{O}_{e_N} \right\rangle
\\
&= \left\langle \left[\mathcal{O}_{e_1}, \mathcal{O}_{e_2}\right] \mathcal{O}_{e_3} \cdots \mathcal{O}_{e_N}\right\rangle  + \left\langle \mathcal{O}_{e_2} \left[\mathcal{O}_{e_1}, \mathcal{O}_{e_3} \right]  \cdots \mathcal{O}_{e_N} \right\rangle + \left\langle \mathcal{O}_{e_2} \mathcal{O}_{e_3} \mathcal{O}_{e_1}  \cdots \mathcal{O}_{e_N} \right\rangle
\\
& \cdots
\\
&= \sum_{i=2}^N \left\langle \cdots \left[\mathcal{O}_{e_1}, \mathcal{O}_{e_i}\right] \cdots \mathcal{O}_{e_N}\right\rangle + \left\langle \mathcal{O}_{e_2} \mathcal{O}_{e_3}   \cdots \mathcal{O}_{e_N} \mathcal{O}_{e_1}\right\rangle
\\
&= \sum_{i=2}^N \left\langle \cdots \left[\mathcal{O}_{e_1}, \mathcal{O}_{e_i}\right] \cdots \mathcal{O}_{e_N}\right\rangle
\end{align}
At each step of the application of this procedure, two new terms appear. Following~\cite{John}, we refer to the term that involves the switching of two operators as the \emph{passing term}, and the terms that involve the commutator as the \emph{canceling term}.
\item[\emph{vii).}] After possible simplifications of the remaining terms, one can define the commutators as new operators and consider each of the finitely many obtained terms as a VEV in at most $N-1$ operators $\mathcal{O}'_{f_1}, \dots, \mathcal{O}'_{f_{N-1}}$ with well-defined energies and charges. The key point of \emph{vi).} is then that starting with one VEV of $N$ operators, one ends up with $N-1$ VEVs of $N-1$ operators each. For each term, the procedure can be repeated, at the ends of which one ends up with at most $(N-1)(N-2)$ nonvanishing VEVs each of which has $N-2$ operators. After $N-1$ iterations one ends up with at most $(N-1)!$ VEVs of a single zero energy operator each: these terms either vanish or are of the form 
\begin{equation}
\left\langle \mathcal{O}_{0} \right\rangle = \left\langle \alpha \cdot \id \right\rangle =  \alpha \left\langle \id \right\rangle = \alpha, \qquad \qquad \text{for some non zero constant } \alpha
\end{equation}
Summing up all such constants leads to the result:
\begin{equation}
\left\langle \mathcal{O}_{e_1} \mathcal{O}_{e_2} \cdots \mathcal{O}_{e_N} \right\rangle = \sum_{j=1}^{M} \alpha_j, \qquad \qquad M \leq (N-1)!
\end{equation}
\item[\emph{viii).}] Notice that the procedure described in \emph{vi).}, \emph{vii).} establishes an explicit algorithm to compute any semi-infinite wedge VEV, provided one can compute the commutators effectively. De facto, the computation gets much easier if the initial operators $\mathcal{O}_{e_i}$ belong to a graded by the energy Lie subalgebra of the algebra of operators with the known commutation relations. In this case we can make the computations faster, by considering applying the rules to the leftmost operator of negative energy and moving it to the right, in general, it reduces the number of terms. In the following we are going to use this version of the algorithm.
\end{itemize}

To visualize the algorithm's operation, we model its steps as a rooted tree. The nodes of this tree are naturally graded by their distance to the root. The tree is constructed inductively by the following procedure:

\begin{itemize}
\item \emph{Root Node:} The root represents the initial VEV.
\item \emph{Induction step:} Every node at a distance 
$n$ from the root, which represents a non-zero VEV of operators with at least one operator having energy different from $0$, has two direct descendants:
\begin{itemize}
\item One descendant corresponds to the canceling term.
\item The other corresponds to the passing term.
\end{itemize}
\item \emph{Termination:} A vertex, representing the VEV which is equal to $0$, or a VEV of several energy zero
operators does not produce any further descendants. This marks the end of the branching at that node. 
\item \emph{Edges labelings:} Edges representing canceling terms are labeled with the structure constant of the operator algebra, derived from the commutator of the corresponding operators.
\end{itemize}

We refer to this tree as the \emph{VEV-computation tree}. The paths in the VEV-computation tree that start in the root and terminate in the vertices labeled by VEVs of products of energy zero operators are referred to as \emph{essential paths.} See Figure~\ref{fig:algorithm} for the details. Given a VEV-computation tree, the initial VEV can be computed as the sum over the set of its essential paths of the products of the labels of vertices multiplied by the VEV in the terminal vertex of the path.

\begin{figure}
\begin{tikzpicture}[
    every node/.style={}, 
    every edge/.style={draw, -latex} 
]
\begin{scope}

\node (root) at (0,0) {$\left\langle \mathcal O_2\mathcal O_1 \mathcal O_{-3} \right\rangle$};

\node [below left=2cm and 1cm of root] (left) {$\left\langle \mathcal O_2 \mathcal O_{-2}\right\rangle$};
\node [below right=2cm and 1cm of root] (right) {$\left\langle\mathcal O_2 \mathcal O_{-3} \mathcal O_1\right\rangle = 0$};

\draw (root) -- node[left] {$p(1,-3)$} (left);
\draw (root) -- node[right] {} (right);

\node [below left=1.5cm and 0.5cm of left] (left-left) {$\left\langle \mathcal O_0 \right \rangle$};
\node [below right=1.5cm and 0.5cm of left] (left-right) {$\left \langle \mathcal O_{-2} \mathcal O_2 \right \rangle = 0$};

\draw (left) -- node[left] {$p(2,-2)$} (left-left);
\draw (left) -- node[right] {} (left-right);


\end{scope}

\end{tikzpicture}
\caption{A VEV-computation tree. We compute the VEV of operators $\mathcal{O}_2,\mathcal O_1,\mathcal O_{-3}$ with the commutation relation $[\mathcal O_{e_1},\mathcal O_{e_2}] = p(e_1,e_2)\mathcal O_{e_1 + e_2}$, where $p$ is some function. In each node, its right son is the corresponding passing term, and its left son is the corresponding canceling term. The edges connecting a vertex with the corresponding canceling term is marked by the value of the function $p$ that appears as the factor. To compute the commutator we have to take the sum over all essential paths of products of all the labels written on the edges of the path, and the label of the bottom-most vertex of the path. In the depicted case, there is only one essential path, so the VEV in question equals $p(1,-3)p(2,-2)\langle \mathcal O_0 \rangle$.
}\label{fig:algorithm}
\end{figure}
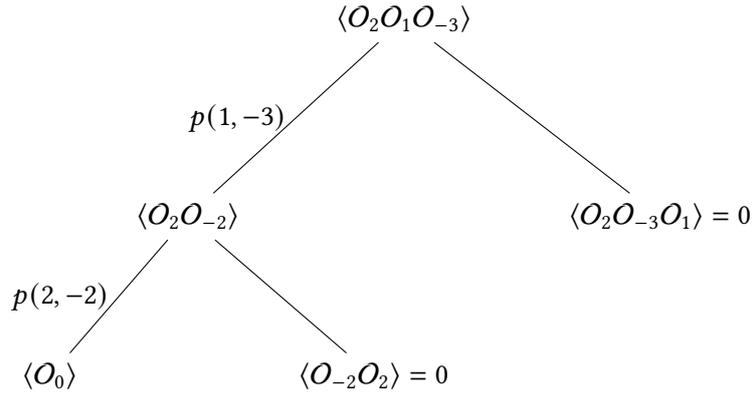

\vspace{1cm}
\section{Leaky Hurwitz numbers as VEVs}
\label{sec:Hurwitz VEVs}

We are now ready to define the Hurwitz numbers of interest in terms of the semi-infinite wedge formalism. Let $k$ be an integer number, $g,s$ be non-negative integers, and $\mu,\nu$ be partitions of sizes $l(\mu)$ and $l(\nu)$ respectively.\par

\begin{definition}
The (possibly disconnected) $k$-leaky $(r+1)$-completed cycles Hurwitz numbers are given by:
\begin{equation}
h_{g;\mu, \nu}^{\wr,k,r} \coloneqq [z_1^{r+1} \cdots z_s^{r+1}].\frac{1}{\prod\mu_i\prod\nu_j}\left\langle \prod_{i=1}^m \alpha_{\mu_i} \prod_{p=1}^s \mathcal{E}_{-k}(z_p) \prod_{j=1}^n\alpha_{-\nu_j}\right\rangle\,,
\end{equation}
where the operator $[x_1^{d_1} \cdots x_N^{d_N}]$ applied to a formal power series in the variables $x_1, \dots, x_N$ selects the coefficient of the monomial $x_1^{d_1} \cdots x_N^{d_N}$ (e.g. $[x^3]. \sum_{i=0} a_i x^i = a_3$) and $s$ is linked with $g, k, n, m$ via the Riemann-Hurwitz linear constraint:
$$
g \coloneqq \frac{rs +2 - m - n}{2} \in \N
$$
and is required to be an integer. For the case $k = 0$ we mean that the operators $\tilde \E_0(z_i)$ are used (see~\cite{SSZ} for explanation). Also, for $k=0$ the number $g$ is the genus of the ramified covering Riemann surfaces which are enumerated by the Hurwitz numbers, see \cite{CM}. We will interchangeably focus either on the genus, using the notation $h_{g;\mu, \nu}^{\wr,k,r}$, or on the number of operators $\E_k(z)$ in the VEV using the notation $h_{\mu, \nu}^{\wr,k,r,s}$.
\end{definition}

As customary in enumerative geometry, the connected and the possibly disconnected counterparts of the same enumerative problem are related by the inclusion-exclusion formula.

\begin{definition}
 The connected $k$-leaky $(r+1)$-completed cycles Hurwitz numbers are defined as
\begin{equation}
h_{g;\mu, \nu}^{\wr,k,r, \circ} \coloneqq [z_1^{r+1} \cdots z_s^{r+1}].\frac{1}{\prod\mu_i\prod\nu_j}\left\langle \prod_{i=1}^m \alpha_{\mu_i} \prod_{p=1}^s \mathcal{E}_{-k}(z_p) \prod_{j=1}^n\alpha_{-\nu_j}\right\rangle^{\circ},
\end{equation}
where the connected correlators $\left\langle \bullet \right\rangle^{\circ}$ are related to the possibly disconnected correlators $\left\langle \bullet \right\rangle$ via the inclusion-exclusion formula:
\begin{equation}\label{eq:inclexcl}
\left\langle 
\mathcal{O}_{(1)} \dots \mathcal{O}_{(N)} 
\right\rangle^{\circ}
\coloneqq
\sum_{M \vdash \{1, \dots, N\}} (-1)^{|M| - 1} (|M| - 1)!
\prod_{i=1}^{|M|}
\left\langle 
\mathcal{O}_{((M_i)_1)} \dots \mathcal{O}_{((M_i)_{\ell(M_i)})} 
\right\rangle ;
\end{equation}
where $M$ is a set partition into subsets with elements ordered increasingly. This formula can be inverted in the usual way into the more familiar inclusion-exclusion formula:
\begin{equation}\label{eq:inclexclinv}
\left\langle 
\mathcal{O}_{(1)} \dots \mathcal{O}_{(N)} 
\right\rangle
=
\sum_{M \vdash \{1, \dots, N\}}
\prod_{i=1}^{|M|}
\left\langle 
\mathcal{O}_{((M_i)_1)} \dots \mathcal{O}_{((M_i)_{\ell(M_i)})}
\right\rangle^{\circ}.
\end{equation}
\end{definition}

\begin{remark}
Setting $k=0$ recovers the usual $(r+1)$-completed cycles Hurwitz numbers studied in \cite{SSZ} (see also references within). Setting furthermore $r=1$ recovers the simple double Hurwitz numbers studied in \cite{John}. Setting $r=1$ and keeping the leaky parameter $k$ arbitrary recovers the simple double $k$-leaky Hurwitz numbers studied in \cite{CMS} up to a correction that is discussed in section \ref{sec:geometry}.
\end{remark}

We introduce the following notation.

\begin{definition}
For any subsets $I\subset [l(\mu)] =\{1, \dots, l(\mu)\}$, $J\subset [l(\nu)]$, and $K\subset [s]=\{1, \dots, s\}$
we denote by $\mu_I$, $\nu_J$, and $z_K$ the corresponding sums
$\mu_I :=\sum_{i\in I}\mu_i$, $\nu_J :=\sum_{j\in J}\nu_j$, and $z_K :=\sum_{\ell \in K} z_\ell$. Let $t$ be an integer.
Define the operators:
\begin{equation}\label{eq:EE}
\E (I,J+t,K) \coloneqq \E_{\mu_I - \nu_J - kt} \left( z_K \right), \qquad \qquad t = 0, \dots, s.
\end{equation}

Let $M\subset [l(\mu)]$, $N\subset [l(\nu)]$, and $L\subset [s]$, we use:
\begin{equation}\label{eq:zz}
\mZeta{I}{J + t}{K}{M}{N + v}{L} 
:= \varsigma \left(\det \left( {\begin{smallmatrix}
 \mu_I - \nu_J - tk & z_K \\
 \mu_M - \nu_N - vk & z_L
 \end{smallmatrix} }\right)\right).
\end{equation}
\end{definition}

The possibly disconnected $k$-leaky completed $(r+1)$-cycles Hurwitz number then can be written as:
\begin{align}\label{algorithmVacuumExpectation}
h_{\mu,\nu}^{\wr,k,r,s} &= \frac{1}{\prod_{i=1}^{l(\mu)} \mu_i \prod_{j=1}^{l(\nu)} \nu_j}\cdot [z_1^{r+1} \cdots z_s^{r+1}] \\
\notag & \left\langle \prod_{i=1}^{l(\mu)}\mathcal{E}(\{i\},\emptyset + 0,\emptyset)  \prod_{\ell=1}^{s} \mathcal{E}(\emptyset,\emptyset + 1,\{ \ell \})
\prod_{j=1}^{l(\nu)}\mathcal{E}(\emptyset,\{j\} + 0,\emptyset)\right\rangle.
\end{align}
Similarly, the connected $k$-leaky completed $(r+1)$-cycles Hurwitz number then can be written as:
\begin{align}\label{algorithmVacuumExpectation}
h_{\mu,\nu}^{\wr,k,r,s, \circ} &= \frac{1}{\prod_{i=1}^{l(\mu)} \mu_i \prod_{j=1}^{l(\nu)} \nu_j}\cdot [z_1^{r+1} \cdots z_s^{r+1}] \\
\notag & \left\langle \prod_{i=1}^{l(\mu)}\mathcal{E}(\{i\},\emptyset + 0,\emptyset)  \prod_{\ell=1}^{s} \mathcal{E}(\emptyset,\emptyset + 1,\{ \ell \})
\prod_{j=1}^{l(\nu)}\mathcal{E}(\emptyset,\{j\} + 0,\emptyset)\right\rangle^{\circ} .
\end{align}

As stated in the previous section, the computations of the VEVs simplifies if the corresponding class of operators is closed under the commutator. We exploit the fact that the operators $\mathcal{E}_k(z)$ form a Lie algebra by equation \eqref{eq:commE}, and use its explicit structure constants. The following Proposition immediately follows from the definitions \eqref{eq:EE} and \eqref{eq:zz}.

\medskip

\begin{proposition}\label{prop:structureconstants}
    
For any subsets $I, M \subset [l(\mu)]$, $J,N \subset [l(\nu)]$ and $K,L \subset [s]$ such that $I \cap M=J \cap N=K \cap L=\emptyset$, we have:
\begin{align}\label{commutationEquation}
 \left[\E(I, J+t, K) ,\E(M, N+v, L) \right]
= \mZeta{I}{J+t}{K}{M}{N+v}{L}
\E\left(I \cup M, J \cup N + t + v, K \cup L\right).
\end{align}
\end{proposition}

\vspace{1cm}
\section{Geometric and tropical interpretation}
\label{sec:geometry}

In this section we briefly comment on the comparison between leaky completed cycles Hurwitz numbers, specialized to the value $r=1$, with the non-completed Hurwitz numbers already studied in the literature, mainly in \cite{CMR} and in \cite{CMS} without descendants. 
More explicitly, in \cite{CMR}[Definition 5.2.1] the following operator is defined:
\begin{equation}
M_k \coloneqq \frac{1}{3!} \sum_{\substack{a+b+c = k \\ a,b,c \neq 0}} :\alpha_a \alpha_b \alpha_c :
\end{equation}
and the non-completed $k$-leaky double Hurwitz numbers for $r=1$ analyzed in the paper are expressed in the Fock space by \cite{CMR}[Theorem 5.2.2] as the VEV:
\begin{equation}
h_{\mu, \nu}^{\wr,k,r=1,s} = \frac{1}{|\Aut(\mu)||\Aut(\nu)|} \frac{1}{\prod \mu_i\prod \nu_j} \left\langle \prod_{i=1}^m \alpha_{\mu_i} M_{-k}^s \prod_{j=1}^n\alpha_{-\nu_j}\right\rangle.
\end{equation}
The automorphism factors purely depend on the geometric conventions adopted and do not play much role in the comparison. Instead, these numbers are essentially the completed Hurwitz numbers specialized to $r=1$ :
\begin{equation}
h_{\mu, \nu}^{\wr,k, r=1,s} = [z_1^{2} \cdots z_s^{2}].\frac{1}{\prod\mu_i\prod\nu_j}\left\langle \prod_{i=1}^m \alpha_{\mu_i} \prod_{p=1}^s \mathcal{E}_{-k}(z_p) \prod_{j=1}^n\alpha_{-\nu_j}\right\rangle,
\end{equation}
with the difference that the completed version carries a "\emph{torus correction}" naturally arising from the Boson-Fermion correspondence \eqref{eq:bosonfermion} for the $\E$-operators:
\begin{equation}
[z^{2}] \mathcal{E}_{-k}(z) = M_{-k} + \frac{(k+1)(k-1)}{24} \alpha_{-k}.
\end{equation}
Notice that this correction is usually undetectable in the non-leaky case, i.e. for $k=0$, since the operator $\alpha_0$ corresponds to the charge operator and all computation takes place in the charge zero sector, hence this term vanishes as soon as it starts interacting with anything else. What is curious though, is that the torus correction is not there for the values of $k=1$ and $k=-1$. It would be interesting to have a more geometric understanding of this fact possibly coming from Gromov-Witten theory of from intersection theory of the moduli space of stable curves.

\vspace{1cm}
\section{VEVs computation and strong piecewise polynomiality}
\label{sec:piecewise}

In this section, we generalise the algorithm for computation of double Hurwitz numbers in~\cite{John}, extended in \cite{SSZ} to completed $(r+1)$-cycles, to the case of leaky completed $(r+1)$-cycles Hurwitz numbers, although it has also been applied to other types of Hurwitz numbers, see e.g. \cite{AKL} for applications to monotone, strictly monotone Hurwitz numbers as well as hypergeometric tau functions. 

The same argument follows through in our case, with minor modifications. This approach allow us to prove the piecewise polynomiality of such Hurwitz numbers in this section, and the wall crossing formula in the next. 

With the same notation as in the previous section, we now apply the algorithm with a minor modification for the computation of VEVs described in section \ref{subsec:computation} to the operators $\mathcal{E}(I,J+t,K)$. This modification allows to reduce the number of terms involved in the algorithm.

For fixed $\mu,\nu,s,k$ introduce the \emph{coarse VEV-computation tree}\footnote{\cite{John,SSZ} use a notion of the set of \emph{commutation patterns}, which is essentially equivalent to it.}, which essentially contains the same information as the VEV-computation tree, but more adjusted to the combinatorics of partitions. The coarse VEV-computation tree is also constructed inductively. We compute the VEV:
\[
\left\langle \prod_{i=1}^{l(\mu)}\mathcal{E}(\{i\},\emptyset + 0,\emptyset)  \prod_{\ell=1}^{s} \mathcal{E}(\emptyset,\emptyset + 1,\{ \ell \})
\prod_{j=1}^{l(\nu)}\mathcal{E}(\emptyset,\{j\} + 0,\emptyset)\right\rangle.
\]
\begin{itemize}
\item \emph{Root Node:} The root represents the initial VEV. It is marked by the sequence:
\begin{align}
\left( (\{1\},\emptyset + 0,\emptyset);\ldots;(\{l(\mu)\},\emptyset + 0,\emptyset);
(\emptyset,\emptyset +1,\{1\});\ldots;(\emptyset,\emptyset + 1,\{s\}); \right. 
\\ 
\left. (\emptyset, \{1\} + 0,\emptyset);\ldots; (\emptyset,\{l(\nu)\} + 0,\emptyset) \right)
\end{align}
of arguments of the initial $\E$-operators.
\item \emph{Induction step:} Every node $p$ at a distance 
$n$ from the root, which represents a non-zero VEV of operators with at least one operator having energy different from $0$ is marked by a sequence of arguments of the corresponding $\E$-operators:
\begin{align}
\left( (I_{p,1},J_{p,1} + t_{p,1},K_{p,1});\ldots ;(I_{p,s(p)},J_{p,s(p)} + t_{p,s(p)},K_{p,s(p)}) \right)
\end{align}
 Let $n_p = min\{j = 1,\ldots,s\,|\, \mu_{I_i} - \nu_{J_i} - kt_i < 0 \}$ -- the position of the leftmost $\E$-operator with negative energy. By the requirement that the corresponding VEV should be non-zero, $n_p>1$. The node $p$ has two direct descendants:
\begin{itemize}
\item The descendant $pc$ corresponds to the canceling term. Its marking is 
\begin{align}
\left( (I_{pc,1},J_{pc,1} + t_{pc,1},K_{pc,1});\ldots; (I_{pc,s(pc)},J_{pc,s(pc)} + t_{pc,s(pc)},K_{pc,s(pc)})\right)  \\ = \left( I_{p,1},J_{p,1} + t_{p,1},K_{p,1}); \ldots; (I_{p,n_p-1}\cup I_{p,n_p},J_{p,n_p - 1} \cup J_{p,n_p} + t_{p,n_p - 1} + t_{p,n_p} ,K_{p,n_p - 1} \cup K_{p,n_p});\right.\\ \left. \ldots; (I_{p,s(p)},J_{p,s(p)} + t_{p,s(p)},K_{p,s(p)}) 
\right)
\end{align}
in agreement with Proposition~\ref{prop:structureconstants}.
\item The descendant $pp$ corresponds to the passing term. Its marking is 
\begin{align}
\left( (I_{pc,1},J_{pc,1} + t_{pc,1},K_{pc,1});\ldots; (I_{pc,s(pc)},J_{pc,s(pc)} + t_{pc,s(pc)},K_{pc,s(pc)})\right)  \\ = \left( I_{p,1},J_{p,1} + t_{p,1},K_{p,1}); \ldots; (I_{p,n_p},J_{p,n_p} + t_{p,n_p},  K_{p,n_p}); (I_{p,n_p-1},J_{p,n_p - 1} + t_{p,n_p-1} ,K_{p,n_p - 1}  );\right.\\ \left. \ldots; (I_{p,s(p)},J_{p,s(p)} + t_{p,s(p)},K_{p,s(p)})\right)
\end{align}
\end{itemize}
\item \emph{Termination:} A vertex, representing the VEV which is equal to $0$, or a VEV of several energy zero
operators does not produce any further descendants. This marks the end of the branching at that node. 
\item \emph{Edges labelings:} The edge connecting the node $p$ and its descendant $pc$ is labeled by the the value of $\varsigma$-function:
\[
\mZeta {I_{n_p - 1}} {J_{n_p - 1} + v_{n_p - 1}} {K_{n_p - 1}} {I_{n_p}}  {J_{n_p} + v_{n_p}}  {K_{n_p}} 
\]
in agreement with Proposition~\ref{prop:structureconstants}. See Figure~\ref{fig:coarse} for an example.
\end{itemize}

\begin{figure}
\begin{tikzpicture}[
    every node/.style={font = \small}, 
    every edge/.style={draw, -latex} 
]
\begin{scope}

\node (root) at (0,0) {$\left\langle \mathcal \E\left(\{1\},\emptyset + 0,0\right)\E(\emptyset,\emptyset + 1,\{1\})\E(\emptyset,\{1\}+0,\emptyset) \right\rangle$};

\node [below left=2cm and -3cm of root] (left) {$\left\langle \mathcal \E(\{1\},\emptyset + 1,\{1\})\E(\emptyset,\{1\}+0,\emptyset) \right\rangle$};
\node [below right=2cm and 0cm of root] (right) {$0$};

\draw (root) -- node[left] {$\mZeta {\{1\}} {\emptyset  + 0} {\emptyset} {\emptyset}  {\emptyset + 1}  {\{1\}} $} (left);
\draw (root) -- node[right] {} (right);

\node [below left=1.5cm and -2cm of left] (left-left) {$\left\langle \E(\{1\},\{1\} + 1,\{1\}) \right \rangle$};
\node [below right=1.5cm and 0.5cm of left] (left-right) {$0$};

\draw (left) -- node[left] {$\mZeta {\{1\}} {\emptyset  + 1} {\{1\}} {\emptyset}  {\{1\} + 0}  {\emptyset}$} (left-left);
\draw (left) -- node[right] {} (left-right);


\end{scope}

\end{tikzpicture}
\caption{A coarse VEV-computation tree constructed in the assumption $k > 0$ and the both $\mu$ and $\nu$ have one part only. Notice, that the same coarse VEV-computation tree works for, say, $\mu = (5), k = 1, \nu = (4)$, and, say, $\mu = (6), k = 3, \nu = (3)$.
}\label{fig:coarse}
\end{figure}
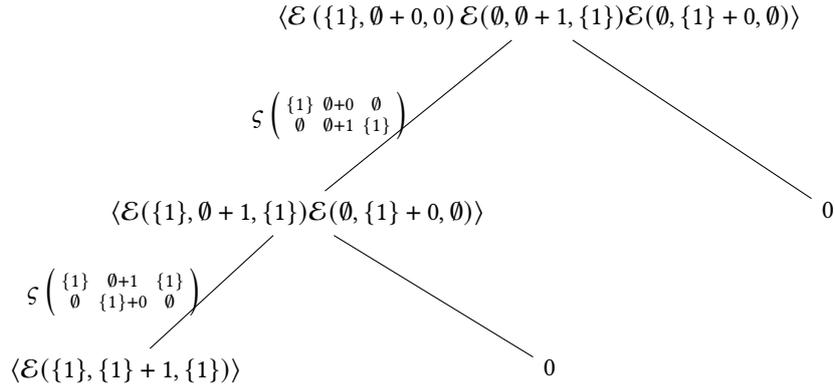

As in the case of the general VEV-computation tree, the initial VEV can be computed by the following expression. Let $EP_{\mu,\nu,k}^s$ be the set of the essential paths in the coarse VEV-computation tree. For $P \in EP_{\mu,\nu,k}^s$ the symbol $T(P)$ denotes the number of factors of energy zero operators in the terminating vertex of $P$, and let $S \vdash [s]$ denote the set partition relative to the variables in each $\mathcal{E}_0$-operator factor in the terminal vertex. The set of edges of the path $P$ that correspond to canceling terms are denoted $L(P)$, and the edge $l\in L(P)$ carries the label
\[
\mZeta{I_l}{J_l + t_l}{K_l}{M_l}{N_l + v_l}{L_l}.
\]
By \emph{i)., ii).}, the only surviving component of a VEV constituted by energy zero operators is their identity component, and by equation \eqref{eq:defE} we have that
\begin{gather}
\left\langle \mathcal{E}_0(z_S)\right. = \frac{1}{\varsigma(z_S)} \left\langle 0\right|, \\ \left.\mathcal{E}_0(z_S) \right\rangle = \frac{1}{\varsigma(z_S)} \left| 0\right\rangle.
\end{gather}
so the VEVs in the terminal vertices an essential path $P$ evaluates to 
\[
\left(\prod_{p = 1}^{T(P)} \frac{1}{\varsigma(z_{S_p})}\right).
\]

It is not hard to see (see e.g. \cite{SSZ}) that the condition $T(P) = 1$ exactly determines the connected correlators via the inclusion-exclusion formula in equation \eqref{eq:inclexclinv}. 
This completes the proof of the following theorem.

\begin{theorem}\label{thm:Hzeta}
Let $\mu,\nu$ be two partitions, and $k,s$ be two natural numbers. The VEV $h^{\wr,k,r,s}_{\mu,\nu}$ can be expressed as the following sum over the essential paths of the corresponding coarse VEV-computation tree:
\begin{align}\label{eq:computationH} 
h_{\mu,\nu}^{\wr,k,r,s} = \frac{1}{\prod_{i} \mu_i \prod_{j} \nu_j}[z_1^{r+1}\cdots z_s^{r+1}]
\sum_{P \in EP_{\mu,\nu,k}^{s}} \left(\prod_{p = 1}^{T(P)} \frac{1}{\varsigma(z_{S_p})}\right) \prod_{l\in L(P)}\mZeta 
{I_l}{J_l + t_l}{K_l}{M_l}{N_l + v_l}{R_l}
\end{align}
Moreover, the connected Hurwitz number $h_{\mu,\nu}^{\wr,k,r, s,\circ}$ is equal to the sum over the essential paths $P$ of the corresponding VEV-computation tree that satisfy $T(P) = 1:$
\begin{align} \label{eq:computationHcirc} 
h_{\mu,\nu}^{\wr,k,r,s,\circ} = \frac{1}{\prod_{i} \mu_i \prod_{j} \nu_j}[z_1^{r+1}\cdots z_s^{r+1}] 
\frac{1}{\varsigma(z_{[s]})} \sum_{P \in EP_{\mu,\nu,k}^{s,\circ}} \prod_{l\in L(P)}\mZeta
{I_l}{J_l + t_l}{K_l}{M_l}{N_l + v_l}{R_l},
\end{align}
where $[s]=\{1, \dots, s\}$.
\end{theorem}

\begin{remark}\label{rem:dual}
    Notice, that we have the following duality. For any $\mu,\nu, k$
    \[
    h^{\wr,k,r,s}_{\mu,\nu} = h^{\wr,-k,r,s}_{\nu,\mu}.
    \]
    This formula is implied by the following considerations. For the left-hand part of the equality we can use the algorithm in its version that moves the negative energy operators to the left, and for the right-hand side we can use the version of the algorithm that moves positive energy operators to the right. For both cases, the corresponding coarse VEV-commutation trees coincide.   
\end{remark}

In the remaining part of this section, we prove two more results about leaky completed cycles Hurwitz numbers: the strong piecewise polynomiality, also referred to as chamber polynomiality, and an analysis on the structure of these polynomials. Both results generalize for arbitrary leaky number $k$ (and recover for $k=0$) the original results of Johnson~\cite{John} for double Hurwitz numbers and their completed cycles generalization in~\cite{SSZ}.

Notice, that the coarse VEV-computation tree does not really contain the information about $\mu,\nu,k$ -- only about the sign of the energy of the operators $\E(I,J+t,K)$ that arise in the process of the algorithm operation.

\begin{definition}\label{def:hyperplane:arrangement}
For fixed natural numbers $n,m,s,r$, define the map 
\begin{align}
\mathfrak{h}: V \coloneqq \left\{(\vec{x},\vec{y},k) \in \N^n \oplus \N^m \oplus \N \,\left|\, \sum_{i=1}^n x_i = \sum_{j=1}^m y_j + s\cdot k \right. \right\} & \longrightarrow \Q
\\
(\mu,\nu,k) & \mapsto \mathfrak{h}(\mu, \nu,k) \coloneqq h^{\wr,k,r, s}_{\mu,\nu}
\end{align}
as well as the hyperplanes $W^{\wr}_{I,J,t}$ of $V$ through the linear equations 
\begin{equation}
W^{\wr}_{I,J,t} : x_I - y_J - kt = 0, \qquad \qquad \emptyset \neq I \subsetneq [m], \;\; \emptyset \neq J \subsetneq [n], \;\; t = 0, \dots, s.
\end{equation}
We call $W^{\wr}_{I,J,t}$ \emph{walls of the hyperplane arrangement}, and the connected components of $V$ minus all walls are called \emph{chambers of the hyperplane arrangement}.
\end{definition}

The partitions $\mu,\nu$ can be interpreted as sequences of natural numbers. So, we can define what it means for a triple $(\mu,\nu,k)$ to belong to a chamber or a wall of the hyperplane arrangement. Notice, that it does not depend on the way we order the parts of $\mu$ and $\nu$. Also notice, that for a pair of triples $(\mu,\nu,k)$ and $(\mu',\nu',k')$ that belong to the same chamber of the hyperplane arrangement, the corresponding coarse VEV-computation trees are identical.

\begin{remark}\label{rem:wall}
    Strictly speaking, we also should have included the hyperplane $k = 0$ to the list of the walls, as the coarse VEV-computation trees for the cases $k<0$ and $k>0$ are different. The reason we don't add this wall to the list will become clear shortly.
\end{remark}

\begin{lemma}\label{lem:chamber}
Let $\mathfrak c$ be a chamber of the hyperplane arrangement for $V$. Then for any $(\mu,\nu,k)\in V$ with $l(\mu) = m,l(\nu) = n,$ that belong to $\mathfrak{c}$ we have
$$
h^{\wr,k,r, s}_{\mu,\nu}=h^{\wr,k,r, s,\circ}_{\mu,\nu}
$$
\begin{proof}
This is a well-known fact from infinite wedge formalism (e.g. \cite{John, SSZ, AKL}). It can be seen from the fact that the equalities that define the walls correspond to possible ways to have energy zero operators throughout the algorithm of computation of the possibly disconnected VEV defining the Hurwitz number. The operator $\mathcal{E}(I, J + t, K)$ has energy 
$$
\mu_I - \nu_J - kt.
$$
So for a triple $(\mu,\nu,k)$ that belongs to a chamber of the hyperplane arrangement, all the terminal vertices for essential paths in the coarse VEV-computation tree correspond to VEVs of a single energy zero operator. The assertion follows.
\end{proof}
\end{lemma}

\begin{theorem}\label{PiecewisePolynomialleaky}
The function $\mathfrak{h} \colon V \to \Q$ is piecewise polynomial with walls $W^{\wr}_{I,J,t}$. 
\end{theorem}

\begin{proof} The proof immediately follows from the ones in~\cite{John, SSZ} without any noticeable modification. So we just demonstrate its main steps.

Consider the coarse VEV-computation tree for the triple $(\mu,\nu,k)$ that belongs to a chamber of the hyperplane arrangement. All the terminal vertices of essential paths of the VEV-computation tree are marked with the single argument:
\[
\left(([l(\mu)],[l(\nu)] + s, [s])\right),
\] as it is the only way to get the energy zero operator for the triple in question. Thus the contribution from the VEV in the terminal vertex equals:
\[\frac 1{\varsigma(z_1 + \cdots + z_s)}.\]
The direct parents of the terminal vertices of the essential path correspond to the splitting of all the sets $[l(\mu)],[l(\nu)],[s]$ into two disjoint subsets. The label of the corresponding edge of the coarse VEV-computation tree is :
\[
\mZeta{I}{J + t}{K}{[l(\mu]\setminus I}{[l(\nu)]\setminus J + s - t}{[s]\setminus K}
\]
Due to the fact that $\mu_{[l(\mu)]} + \nu_{l[\nu]} - ks = 0$, the argument of the $\varsigma$-function arising from the commutator is $az_{[s]}$ for some integer $a$. Moreover, as $(\mu,\nu,k)$ belong to the chamber, $a\ne 0.$ As $\frac{\varsigma(az)}{\varsigma(z)}$ is holomorphic in 0, the resulting power series is holomorphic in variables $z_1,\ldots,z_s.$

The fact the resulting expression is also divisible by $\prod \mu_i \prod \nu_j$  follows from the fact that the commutators:
\[
[\alpha_{\mu_p}, \E(I,J+t,K)] = \varsigma(\mu_p z_K) \E(I\cup \{p\}, J+ t, K)
\]
and
\[
[\E(I,J+t,K),\alpha_{-\nu_r}] = \varsigma(\nu_r z_K)\E(I,J\cup\{r\} + t, K).
\]
The coefficients of both appearing $\varsigma-$function are polynomials in $\mu_p$ and $\nu_r$ with vanishing free terms respectively.

We can also notice, that the dependence of $k$ in the resulting expression is also polynomial. However,  as 
\[
k = \frac 1s\left( \mu_{[l(\mu)]} - \nu_{[l(\nu]}\right)
\]
we can also eliminate $k$ from the expressions for VEV completely.

\end{proof}

Now we can return to Remark~\ref{rem:wall} and explain, why we do not include the hyperplane $k = 0$ to the list of walls. Let $k = 0$. Consider the VEV 
\[
\left\langle \alpha_\nu\prod_i \tilde\E_0(z_i) \alpha_{-\mu} \right\rangle
\]
Notice, that it is possible to apply the algorithm in two ways: either treating the operators $\tilde \E_0(z_i)$ similarly to positive energy operators or treating them as negative energy operators. The resulting coarse VEV-computation trees are different. However, Theorem~\ref{PiecewisePolynomialleaky} states that the corresponding VEV is a polynomial (and thus a holomorphic function) in its parameters on both sides of the hyperplane $k= 0$, with the values coinciding on it. By the analytic continuation principle, it means, that the corresponding polynomial is the same on both sides.

We also provide the following analysis of the polynomial structure.

\begin{proposition} \label{PolyChambers}
Let $\mathfrak{c}$ be a chamber of the hyperplane arrangement. Then 
$
\mathfrak{h}|_\mathfrak{c}(\mu,\nu,k)
$
is a polynomial of degree $(r+1)s + 1 - l(\mu) - l(\nu)$ in variables $\mu_i,\nu_j$. Moreover,
$
\mathfrak{h}|_\mathfrak{c}(\mu,\nu,k)
$
is a polynomial in $k$ of degree $(r+1)s + 1$.
 
\begin{proof}
The proof of the analog result for non-leaky completed cycles in \cite{SSZ} implies that:
\begin{equation}
\mathfrak{h}|_\mathfrak{c}(\mu,\nu,k) = \sum_{\ell=0}^g (-1)^\ell P_{\mathfrak{c},\ell}^{\wr, s}(\mu,\nu, k) ,
\end{equation}
where $P_{\mathfrak{c},\ell}^{\wr, s}$ is a homogeneous polynomial in variables $\mu_i,\nu_j$  of degree 
$(r+1)s + 1 - l(\mu) - l(\nu) - 2\ell$ with $P_{\mathfrak{c},\ell}^{\wr,s}(\mu,\nu, k) > 0$ for all $(\mu,\nu) \in \mathfrak{c}$, and 
$g=(rs - l(\mu) - l(\nu) + 2)/2.$
 The degree in $k$ is equal to the degree in $\mu, \nu$, augmented by the amount of $l(\mu) + l(\nu)$ since the division by the factor of $\prod_i \mu_i \prod_j \nu_j$ affects only the $\mu, \nu$ dependency but not the $k$ dependency.
\end{proof}
\end{proposition}

\vspace{1cm}
\section{Wall crossing formulae}
\label{sec:wallcrossing}

In this section we obtain the wall crossing formula for leaky double Hurwitz numbers with completed cycles, that is, we compute the quantity:
$$
WC^{\wr, r}_{I,J,t} \coloneqq \mathfrak{h}|_{\mathfrak{c}_1} - \mathfrak{h}|_{\mathfrak{c}_2}, \qquad \qquad I \subset [m], J \subset [n], t = 0, \dots, s;
$$
where $\mathfrak{c}_1$ and $\mathfrak{c}_2$ are the two chambers bordering the wall:
$$
V \supset W_{I,J,t} : |\mu_I| - |\nu_J| - tk = 0
$$ 
of the hyperplane arrangement, and $I,J,t$ will be fixed throughout the section. One nice feature of Hurwitz wall crossing formulae is that they are usually expressed in terms of (sums of) quadratic products of Hurwitz numbers themselves, and the leaky case makes no exception. Similarly to the rest of the paper, this result generalizes the wall crossing formula obtained in \cite{SSZ} in the $k=0$ case, which itself generalizes the original wall crossing formula for double Hurwitz numbers obtained in \cite{CJM11}. Also in this case, the proof itself follows the same strategy with minor adaptations reflecting the different definitions.

To start, let us define the key quantity:
\begin{equation}
\delta \coloneqq |\mu_I| - |\nu_J| - tk.
\end{equation}
Notice that $\delta$ can be thought as a function on $V$ for fixed $t$, that the wall described precisely the zero locus of the function $\delta$ and that the sign of $\delta$ determines on which side of the wall a point $(\vec{x}, \vec{y},k)$ is located. In other words, $\delta$ is positive if evaluated on $\mathfrak{c}_1$, negative if evaluated on $\mathfrak{c}_2$, and zero if and only if evaluated on the wall $W_{I,J,t}$.

Following~\cite{SSZ} we define the series:
\begin{equation}\label{WallCrossingSeries}
H_{\mu,\nu,k}^{\wr,s}(z_1,\dots,z_s) := \frac{1}{\prod_{i=1}^{n} \mu_i \prod_{j=1}^{m} \nu_j} \left \langle \prod_{i=1}^{n} \mathcal{E}_{\mu_i}(0) \prod_{p=1}^s \mathcal{E}_{-k}(z_p)\prod_{i=1}^{m} \mathcal{E}_{-\nu_j}(0) \right\rangle.
\end{equation}
The expansion coefficients of this series are polynomial in any chamber.

We need the following two specializations in terms of $\delta$: the one with $\delta$ as one of the parts:
\begin{align}\label{eq:Hdelta}
H^{\wr,|K|}_{\mu_I,\nu_J + \delta}(\{z_p\}_{p \in K}) 
=
\frac{1}{\delta \prod_{i \in I} \mu_i \prod_{j \in J} \nu_j} 
\left\langle \prod_{i \in I}\mathcal{E}(\{i\},\emptyset+0,\emptyset)\prod_{p \in K} \mathcal{E}(\emptyset,\emptyset+1,\{p\})\right.\\ \left. \prod_{j \in J}\mathcal{E}(\emptyset,\{j\}+0,\emptyset) \mathcal{E}_{-\delta}(0)\right\rangle,
\end{align}
and its complement:

\begin{align}\label{eq:Hdeltac}
H^{\wr,|K^c|}_{\mu_{I^\text{c} + \delta},\nu_{J^\text{c}}}(\{z_p\}_{k \in K^\text{c}}) 
=
\frac{1}{\delta \prod_{i \notin I} \mu_i \prod_{j \notin J} \nu_j} 
\left \langle \mathcal{E}_\delta(0)\ \prod_{i \notin I}\mathcal{E}(\{i\},\emptyset + 1,\emptyset)\prod_{p \notin K} \mathcal{E}(\emptyset,\emptyset+1,\{p\})\right. \\ \left.\prod_{j \notin J}\mathcal{E}(\emptyset,\{j\}+0,\emptyset) \right\rangle.
\end{align}

\begin{theorem}
The wall crossing formula is given by:
\begin{align}\label{eq:wc}
WC_{I,J,t}^{\wr, r}(\mu,\nu) = 
 [z_1^{r+1} \cdots z_s^{r+1}] & \sum_{\begin{smallmatrix}K \subset [s]\\ |K| = t \end{smallmatrix}} \delta^2 \frac{\varsigma(z_K)\varsigma(z_{K^c})\varsigma(\delta z_{[s]})}{\varsigma(\delta z_K) \varsigma(\delta z_{K^c})\varsigma(z_{[s]})} H_{\mu_I, \nu_J + \delta}^{\wr,|K|}(\{z_p\}_{p \in K}) H_{\mu_{I^c} + \delta, \nu_{J^c}}^{\wr,|K^c|}(\{z_p\}_{k \in K^c}). 
\end{align}
where
\begin{equation}
\delta = \mu_I - \nu_J - t\cdot k.
\end{equation}
\end{theorem}

\begin{proof} 
Here, we consider the case $k<0$ so that all intermediate operators $\E(\emptyset,\emptyset + 1,\{p\})$ have positive energy. The case $k>0$ can be treated using Remark~\ref{rem:dual} then.

Let $EP(\mathfrak{c}_i)$ the set of essential paths within the chamber $i \in \{1,2\}$. Since we are interested in the quantity:
$$
WC_{I,J,t}^{\wr, r}(\mu,\nu) = \mathfrak{h}|_{\mathfrak{c}_1} - \mathfrak{h}|_{\mathfrak{c}_2},
$$
it suffices to track down those essential paths that are produced in ${\mathfrak{c}_1}$ and not in ${\mathfrak{c}_2}$, and vice versa. For this purpose let us observe that an essential path $P \in EP(\mathfrak{c}_1)$ belongs to $EP(\mathfrak{c}_2)$ as well (and vice versa) unless it produces at some point the operator:
$$
\mathcal{E}(I,J+t,K), \qquad \qquad \text{ for some } K \subset [s],
$$
which is by definition of the wall is the only operator that differs in energy sign between the two chambers. If $P$ is a pattern producing the operator $
\mathcal{E}(I,J+t,K)$, we want to start the commuting algorithm by rearranging slightly the operators of the VEV to better fit the considered set $I,J,K$ as:
\begin{align}
\left\langle \prod_{i \notin I} \mathcal{E}(\{i\},\emptyset + 0,\emptyset)  \prod_{i \in I} \mathcal{E}(\{i\}, \emptyset + 0,\emptyset) \prod_{p=1}^s \mathcal{E}(\emptyset, \emptyset+1, \{p\})\right.\\ \left.  \prod_{j \in J}\mathcal{E}(\emptyset,\{j\} + 0,\emptyset)\prod_{j \notin J}\mathcal{E}(\emptyset,\{j\} + 0,\emptyset) \right\rangle. 
\end{align}
This can be done without loss of generality since by definition the operators $\mathcal{E}(\{i\},\emptyset+0,\emptyset) = \alpha_{\mu_i}$ with $\mu_i > 0$ and $[\alpha_{i}, \alpha_{j}]  = i\delta_{i+j}$, so that all commutators $\mathcal{E}(\{i\},\emptyset +0,\emptyset)$ commute by themselves and so do the $\mathcal{E}(\emptyset,\{j\}+0,\emptyset)$.

Any essential path of the coarse VEV-computation tree that involves $\mathcal{E}(I,J+t,K)$ should contain canceling terms that produce the unions of subsets of $I,J,K$.

Let us now consider the following two quantities.
\begin{itemize}
\item Let $T_1^K$ be the product of $\varsigma$-functions produced by the algorithm up until the appearance of $\E(I,J+t,K)$. Note that it does not depend on whether we run the algorithm on $\mathfrak{c}_1$ or~$\mathfrak{c}_2$.
\item Let $T_2^K$ denote the difference between the contributions of essential paths containing $\E(I,J+t,K)$ on $\mathfrak{c}_1$ and~$\mathfrak{c}_2$.
\end{itemize}

We can express the wall-crossing formula in terms of these quantities as
\begin{equation}
WC_{I,J,t}^{\wr, r}(\mu, \nu) = \frac{1}{\prod_{i=1}^m \mu_i \prod_{j=1}^n \nu_j} [z_1^{r+1} \cdots z_s^{r+1}] \sum_{K \subset [s]} T_1^K T_2^K .
\end{equation}
We are thus left with computing $T_1^K$ and $T_2^K$.

\subsubsection{Computation of $T_1^K$}

 It is easy to see that the corresponding product of $\varsigma$-functions is given by: 
\begin{equation}
T_1^K(\{z_p\}_{p \in K}) := \left\langle \prod_{i \in I}\mathcal{E}(\{i\},\emptyset + 0 ,\emptyset)\prod_{p \in K} \mathcal{E}(\emptyset,\emptyset+1,\{p\}) \prod_{j \in J}\mathcal{E}(\emptyset,\{j\}+0,\emptyset) \mathcal{E}_{-\delta}(0)\right\rangle \frac{\varsigma(z_K)}{\varsigma(\delta z_K)} . 
\end{equation}
Therefore,
\begin{equation}\label{eq:t1}
T_1^K(\{z_k\}_{k \in K}) = \delta \prod_{i \in I} \mu_i \prod_{j \in J} \nu_j \frac{\varsigma(z_K)}{\varsigma(\delta z_K)} H_{\mu_I,\nu_J + \delta}(\{z_k\}_{k \in K}).
\end{equation}

\subsubsection{Computation of $T_2^K$}

 To compute $T_2$, it will be better first to 
  move the operator $\mathcal{E}(I, J+t, K)$ to the left on both chambers, even though it has positive energy on~$\mathfrak{c}_1$.  After it is moved to the far left, resume the normal operation of the algorithm.
 
Notice, that all the canceling terms that appear in the process of this movement of the operators, and in the subsequent operation of the algorithm have the same energy signs both in $\mathfrak c_1$ and $\mathfrak c_2$, so the corresponding paths do not contribute to $T_2^K$. The only contribution to $T_2^K$ comes from paths with $\mathcal{E}(I, J+t, K)$ is moved entirely from to left, and the remaining operators are all to the right of it. The resulting paths will then contribute $0$ on~$\mathfrak{c}_2$ since there is an operator of negative energy on the far left, but it will be non-zero on~$\mathfrak{c}_1$. The last step in the algorithm on $\mathfrak{c}_1$ can be computed by noticing that the total energy of the two operators involved should be zero, resulting in:
\begin{equation}
\langle \mathcal{E}(I, J+t, K)\ \mathcal{E}(I^\text{c}, J^\text{c}+(s - t), K^\text{c})\rangle = \frac{\varsigma(\delta z_{[s]})}{\varsigma(z_{[s]})} ,
\end{equation}
where $I^\text{c}$ denotes the complement of $I \subset [m]$, and the same for $J^\text{c}$ and~$K^\text{c}$. Also using that
\begin{equation}
\langle \mathcal{E}_\delta(0)\ \mathcal{E}(I^\text{c}, J^\text{c} + (s-t), K^\text{c}) \rangle = \frac{\varsigma(\delta z_{K^\text{c}})}{\varsigma(z_{K^\text{c}})},
\end{equation}
we see that 
\begin{align}
& T_2^K(\{z_k\}_{p \notin K}) \\ \notag 
&= \langle\mathcal{E}_\delta(0) \prod_{i \notin I}\mathcal{E}(\{i\},\emptyset +0,\emptyset)\prod_{p \notin K}\mathcal{E}(\emptyset, \emptyset+1, \{p\})\prod_{j \notin J} \mathcal{E}(\emptyset, \{j\}+0, \emptyset)\rangle \frac{\varsigma(z_{K^\text{c}}) \varsigma(\delta z_{[s]})}{\varsigma(\delta z_{K^\text{c}})\varsigma(z_{[s]})} .
\end{align}

therefore
\begin{align}\label{eq:t2}
& T_2^K(\{z_k\}_{k \notin K}) \\ \notag
& = \delta \prod_{i \notin I} \mu_i \prod_{j \notin J} \nu_j 
\varsigma(z_{K^\text{c}}) \varsigma(\delta z_{[s]})\varsigma(\delta z_{K^\text{c}})\varsigma(z_{[s]})
H_{\mu_{I^\text{c}} + \delta, \nu_{J^\text{c}}}(\{z_k\}_{k \in K^\text{c}}) .
\end{align}

Substituting the obtained formulae for $T_1^K$ and $T_2^K$ we get the desired result. This concludes the proof of the theorem.
\end{proof}

\subsection{Wall crossing formulae in genus zero}

In the following we specialize the wall crossing formulae down to the $r=1$ case and the $g=0$ case, to recover and compare with the wall crossing formulae for the non-completed $k$-leaky Hurwitz numbers considered in \cite{CMS}.

Recall that in genus zero only the leading term of the $\varsigma$-functions contribute and that furthermore by Riemann-Hurwitz formula we must have $s = n + m - 2$. Moreover, $|K|$ is forced to be equal to $|I| + |J| - 1$, for otherwise, the Hurwitz number considered vanishes. Under these constraints the quantity:

\begin{align}
WC_{I,J,t}^{\wr,k, r}(\mu,\nu) = 
 [z_1^{r+1} \cdots z_s^{r+1}] & \sum_{K \subset [s]} \delta^2 \frac{\varsigma(z_K)\varsigma(z_{K^c})\varsigma(\delta z_{[s]})}{\varsigma(\delta z_K) \varsigma(\delta z_{K^c})\varsigma(z_{[s]})} H_{\mu_I, \nu_J + \delta}^{\wr,|K|}(\{z_p\}_{p \in K}) H_{\mu_{I^c} + \delta, \nu_{J^c}}^{\wr,|K^c|}(\{z_p\}_{k \in K^c}),
\end{align}
specializes to:

\begin{align}
WC_{I,J,t}^{\wr,k, r}(\mu,\nu) &= [z_1^{2} \cdots z_s^{2}] \sum_{|K| = |I| + |J| - 1} \delta^2 \frac{z_K z_{K^c} \delta z_{[s]}}{\delta z_K \delta z_{K^c} z_{[s]}} H_{\mu_I, \nu_J + \delta}^{\wr,|I| + |J| - 1}(\{z_p\}_{p \in K}) H_{\mu_{I^c} + \delta, \nu_{J^c}}^{\wr,|I^c| + |J^c| - 1}(\{z_p\}_{k \in K^c})
\\
&= [z_1^{2} \cdots z_s^{2}] \sum_{|K| = |I| + |J| - 1} \delta \cdot H_{\mu_I, \nu_J + \delta}^{\wr,|I| + |J| - 1}(\{z_p\}_{p \in K}) H_{\mu_{I^c} + \delta, \nu_{J^c}}^{\wr,|I^c| + |J^c| - 1}(\{z_p\}_{k \in K^c})
\\
&= \binom{n+m - 2}{|I| + |J| - 1} \delta \cdot h^{\wr,k,r=1, |I| + |J| - 1}_{\mu_I, \nu_J + \delta} h^{\wr,k,r=1, |I^c| + |J^c| - 1}_{\mu_{I^c} + \delta, \nu_{J^c}};
\end{align}
where the last step follows from the fact that there are $\binom{n+m - 2}{|I| + |J| - 1}$ ways to choose a subset of $|I| + |J| - 1$ elements from one of $n+m-2$ elements and that all coefficients of $[z_1^{2} \cdots z_s^{2}]$ do not depend on the subset $K$ but only on its cardinality $|K|$, which is the same for all subsets considered. The last step of the computation coincides with the formula established in \cite{CMS}.

\subsection{One-part leaky Hurwitz numbers} The remaining part of the section is dedicated to the one-part case. We consider the case $k>0$. 

\begin{proposition}\label{prop:onepart} For $n=1$ we have:
\begin{equation}
h_{d,\nu}^{\wr,k,r,s,\circ} = \frac{1}{d \prod_{j} \nu_j}[z_1^{r+1}\cdots z_s^{r+1}] 
\frac{1}{\varsigma(z_{[s]})} 
\prod_{p=1}^s \varsigma\left((d - pk)z_p + kz_{[p]}\right)
\prod_{j=1}^m \varsigma\left(z_{[s]} \nu_j \right).
\end{equation}
\begin{proof}
Direct computation from equation \ref{eq:computationHcirc}. There is a single positive energy operator in the entire VEV, so starting commuting it to the right all passing terms vanish with only the canceling term surviving, up until the last commutator that gives the result. So there's a single essential path contributing to the result.
\end{proof}
\end{proposition}

\begin{remark}
We stress here that the case $m=1$ is more involved than the $n=1$ case, in contrast with the nonleaky case in which there's a perfect symmetry between the parts of the two side partitions $\mu$ and $\nu$.
\end{remark}

We now compare a very particular case of our formula with the expression appearing in \cite{CMS}[pag 11]. To establish a match between the formulae, we need to specialize the formulae above to the following three conditions: \textit{generic ramifications}, \textit{genus zero}, and $1$-leaky Hurwitz numbers.
The first specialization amounts to set $r=1$, the second to set $s=m-1$, and the third to $k=1$. Note that the leaky $2$-completed cycles differ by a correction factor discussed in section \ref{sec:geometry}, although this correction does not show up when at least one of the following conditions is verified: the genus is equal to zero or the leaky parameter is $k=1$. Since both of them are verified, a fortiori we can forget about this correction for the sake of the comparison. The formula appearing in \cite{CMS} reads

\begin{equation}\label{eq:CMSHurwitz}
h_{d,\nu}^{\wr,1,r=1, m-1, \circ} = \frac{(m-1)!}{2^{m-2}} (2d - 1)(2d - 2)\dots (2d - (m-2)).    
\end{equation}

\begin{proposition} \label{prop:onepartexpl} For $n=1$, $r=1$ and $s=m-1$ the $k$-leaky one part $1$-completed (same as uncompleted in this case) Hurwitz numbers read:
\begin{equation}
h_{d,\nu}^{\wr,k,r=1, s=m-1, \circ} = \frac{(m-1)!}{2^{m-2}}
\Big( 2d - k \Big)
\Big( 2d - 2k \Big)
\cdots
\Big( 2d - (m-3)k \Big)
\Big( 2d - (m-2)k \Big).
\end{equation}

\begin{proof}
The first step consists in extracting the leading term from the $\varsigma$-functions, which is a special consequence of the genus zero constraint. Surprisingly, this extraction simplifies the $\nu_j$ parts against the denominator and the resulting formula becomes completely independent on the partition $\nu$ but for its side (i.e. the degree $d$) and its length, which by Riemann Hurwitz are purely combinatorial data.
\begin{align}
h_{d,\nu,k}^{\wr,r=1,s=m-1, \circ} &= \frac{1}{d \prod_{j} \nu_j}[z_1^2\cdots z_{m-1}^2] 
\frac{1}{\varsigma(z_{[m-1]})} 
\prod_{p=1}^{m-1} \varsigma\left((d - pk)z_p + kz_{[p]}\right)
\prod_{j=1}^m \varsigma\left(z_{[m-1]} \nu_j \right)
\\
&= \frac{1}{d}[z_1^2\cdots z_{m-1}^2] 
\prod_{p=1}^{m-1}\left((d - pk)z_p + kz_{[p]}\right)
\left(z_{[m-1]}\right)^{m-1}.
\end{align}
 To extract the right monomial from the above expression, let us define the quantity:
\begin{equation}
    H(m-1) \coloneqq \frac{1}{d}[z_1^2\cdots z_{m-1}^2] 
\prod_{p=1}^{m-1}\left((d - pk)z_p + kz_{[p]}\right)
\left(z_{[m-1]}\right)^{m-1}.
\end{equation}
By differentiating twice with respect to $z_{m-1}$ one obtains the recursion :
$$
H(m-1) = \frac{m-1}{2}\left( 2d - (m-2)k \right)H(m-2),
$$
as well as the initial condition :
$$
H(1) = 1.
$$
Iterating the recursion $m-2$ times and substituting the initial condition yields to the result.
\end{proof}
\end{proposition}

\begin{remark}
    Specializing the above proposition to $k=1$ one immediately recovers formula \ref{eq:CMSHurwitz}. The formula can also be written in terms of generalized Stirling numbers of the first kind.
    Similar formulae appear in \cite{BR, CP, GT}.
\end{remark}

\begin{remark}
 The formulae above in proposition \ref{prop:onepartexpl} can be generalized to arbitrary higher genus and to general $(r+1)$-completed cycles (or even uncompleted, by removing the necessary corrections) by handling the combinatorics of coefficient extraction in proposition \ref{prop:onepart}.
\end{remark}

\subsection{Vanishing in genus zero}

We conclude this section with one more link with the existing literature. In \cite{CMS} the vanishing of leaky Hurwitz numbers in genus zero with descendant is completely characterised. By imposing the absence of descendants in that paper and specialising the numbers in this paper to $r=1$ and $g=0$, the result in \cite{CMS} can be stated as follows. The number:
$$
h_{\mu, \nu}^{\wr,k,r=1, s=m+n-2, \circ}
$$
vanishes if and only if there exists a positive integer $K$ such that
\begin{equation}\label{eq:vanishingcond}
k = 2K, \qquad \mu_i = K\tilde{\mu_i}, \qquad \nu_i = K\tilde{\nu_i}, \qquad \qquad \text{ for all } i=1, \dots, n; \; j=1, \dots, m.
\end{equation}
We observe that under the conditions listed in \eqref{eq:vanishingcond} the leaky Hurwitz number simplifies, factorising as a constant times a leaky Hurwitz number with leaky parameter $k=2$. 
\begin{equation}
h_{\mu, \nu}^{\wr,k,r=1, s=m+n-2, \circ} = K^{s-1} \cdot h_{\tilde{\mu}, \tilde{\nu}}^{\wr,k'=2,r=1, s=m+n-2, \circ}
\end{equation}
This can be seen by specialising the connected version of \cref{thm:Hzeta} to $r=1$ and genus zero (i.e. selecting the leading term from all $\varsigma(x)$ functions). Since we are extracting the coefficient of $[z_1^2 \cdots z_{m+n-2}^2]$ out of the expression, the number of factors in the product on $L(P)$ must be equal to $2m + 2n -3$, which produces a prefactor of $K^{2n+2m-3}$ that simplifies against a prefactor of $K^{-m-n}$ arising from the denominator $\prod_i \mu_i \prod_j \nu_j$ after factorising out $K$ from each of the parts.
This can be regarded as proof that, in genus zero, leaky Hurwitz numbers are homogeneous polynomials in $(\mu,\nu,k)$, recall \ref{PolyChambers}. In fact, any integer could take the place of $2$ in the argument above.

\vspace{1cm}
\section{Cut-and-join operators}
\label{sec:cutnjoin}

The goal of this section is to provide a $k$-leaky $(r+1)$-completed cycles cut-and-join operator, which is given by the following definition.

\begin{definition}
For any natural number $k \geq 0$ let us define the $k$-leaky $(r+1)$-completed cycles cut-and-join operator $Q^{\wr}_{k,r+1}$ by:
\begin{equation}\label{def:cutandjoin}
\sum_{r \geq -1} Q^{\wr}_{k,r+1} z^{r+1} 
= \frac{1}{\varsigma(z)}\sum_{n=0}^\infty \frac{1}{n!} \sum_{\substack{k_1,\dots,k_m>0 \\ \ell_1, \dots, \ell_{n-m} < 0 \\ \sum k_i = \sum \ell_j + k} } \prod_{i=1}^m \varsigma (k_i z) \frac{p_{k_i}}{k_i} \prod_{j=1}^{n-m} \varsigma (\ell_j z) \frac{\partial}{\partial p_{\ell_j}}.
\end{equation}
\end{definition}

Introduce the following generating series for the leaky disconnected Hurwitz numbers with completed cycles:
\begin{align} 
& \mathcal{H}^{\wr,k, r+1}(\beta, \vec{p}, \vec{q})  \coloneqq
\sum_{\substack{n,m = 1 \\ s = 0}} \; \sum_{
\substack{\mu_1, \dots, \mu_m = 1\\ \nu_1, \dots, \nu_n = 1}
}
h^{\wr,k, r+1}_{g;\mu,\nu} \; \frac{\beta^s}{s!} \;
\frac{p_{\mu_1} \cdots p_{\mu_m}}{m!} \;
\frac{q_{\nu_1} \cdots q_{\nu_n}}{n!}.
\end{align}
As before $s$ is determined by the Riemann-Hurwitz constraint as $s = \frac{2g-2+m+n}{r}$.

\begin{theorem}\label{thm:cutandjoin}
The following partial differential equation holds
\begin{equation}
\frac{\partial}{\partial \beta} \mathcal{H}^{\wr,k, r+1} = Q^{\wr}_{k,r+1} \mathcal{H}^{\wr,k, r+1}  
\end{equation}

\begin{proof}
The proof essentially follows immediately from the proof of Boson-Fermion correspondence \eqref{eq:bosonfermion} under the map from the charge zero sector of the semi-infinite wedge to the algebra of power series 
\begin{align}
\Phi: \mathcal{V}_0 & \longrightarrow \C\left[\left[p_i, \frac{\partial}{\partial p_i}\right]\right]_{i \in \N}
\\
\alpha_n & \mapsto p_n \qquad \text{ if } n > 0
\\
\alpha_n & \mapsto n\frac{\partial}{\partial p_n} \qquad \text{ if } n < 0
\end{align}
In fact is immediately follows that by equation \eqref{eq:bosonfermion} and definition \ref{def:cutandjoin} under the map $\Phi$ one has
\begin{equation}
\Phi\left( [w^{r+1}]. \E_{-k}(w) \right) = Q^{\wr}_{k,r+1}.
\end{equation}
We however provide for the reader's convenience a more detailed proof following \cite{SSZ} and employing some lemmata proved there. First, we can restate the generating series naturally as a VEV, as :
\begin{equation} \label{Eq:H}
\mathcal{H}^{\wr,k, r+1}(\beta, \vec{p}, \vec{q}) 
=
 \left\langle 
 \exp\left(\sum_{\ell \geq 1} \frac{p_\ell \alpha_\ell}{\ell} \right) 
\exp\left(\beta [z^{r+1}]\E_{-k}(z)\right) 
\exp\left(\sum_{v \geq 1} \frac{q_v \alpha_{-v}}{v} \right) 
\right \rangle.
\end{equation}
so that:
\begin{align} \label{Eq:dhdbeta}
\frac{\partial \mathcal{H}^{\wr,k, r+1}}{\partial \beta} 
=
 \left\langle 
  \exp\left(\sum_{\ell \geq 1} \frac{p_\ell \alpha_\ell}{\ell} \right) 
 [w^{r+1}]\E_{-k}(w) 
 \exp(\beta [z^{r+1}]\E_{-k}(z)) 
\exp\left(\sum_{v \geq 1} \frac{q_v \alpha_{-v}}{v} \right) 
 \right\rangle,
\end{align}
The first trick consists in multiplying and dividing by the factor $\exp\left(\sum_{\ell \geq 1} \frac{p_\ell \alpha_\ell}{\ell} \right)$ so that the RHS becomes:
\begin{align} \label{Eq:dhdbeta2}
 [w^{r+1}] \left\langle 
  \exp\left(\sum_{\ell \geq 1} \frac{p_\ell \alpha_\ell}{\ell} \right) 
\E_{-k}(w) 
\exp\left(-\sum_{\ell \geq 1} \frac{p_\ell \alpha_\ell}{\ell} \right) 
\exp\left(\sum_{\ell \geq 1} \frac{p_\ell \alpha_\ell}{\ell} \right) 
 \exp(\beta [z^{r+1}]\E_{-k}(z)) 
\exp\left(\sum_{v \geq 1} \frac{q_v \alpha_{-v}}{v} \right) 
 \right\rangle.
\end{align}
This insertion is functional to the use of the standard Hadamard lemma in Lie theory:
$$
e^{A} B e^{-A} = \sum_{\ell=0} \frac{1}{\ell!}\underbrace{ [A, \dots [A,[A,B]] \dots]}_{\ell \text{ commutators}},
$$
applied to the first three factors of the VEV. By \cite{SSZ}[Lemma 5.4.] one expresses the RHS as:
\begin{align} \label{Eq:dhdbeta3}
 [w^{r+1}] 
 \sum_{n=0}\frac{1}{n!} \sum_{k_1, \dots, k_n =1 } \varsigma(wk_i) \frac{p_{k_i}}{k_i}
 \left\langle 
\E_{-k+\sum_{k_i}}(w) 
\exp\left(\sum_{\ell \geq 1} \frac{p_\ell \alpha_\ell}{\ell} \right) 
 \exp(\beta [z^{r+1}]\E_{-k}(z)) 
\exp\left(\sum_{v \geq 1} \frac{q_v \alpha_{-v}}{v} \right) 
 \right\rangle.
\end{align}
By \cite{SSZ}[Lemma 5.6] one can get rid of the leftmost operator replacing it with a differential operator acting on the VEV:
\begin{align} \label{Eq:dhdbeta4}
 [w^{r+1}] 
 \sum_{n=0}\frac{1}{n!} \sum_{k_1, \dots, k_n =1 } \varsigma(wk_i) \frac{p_{k_i}}{k_i} \mathcal{O}_{-k+\sum_{k_i}}(w) 
 \left\langle 
\exp\left(\sum_{\ell \geq 1} \frac{p_\ell \alpha_\ell}{\ell} \right) 
 \exp(\beta [z^{r+1}]\E_{-k}(z)) 
\exp\left(\sum_{v \geq 1} \frac{q_v \alpha_{-v}}{v} \right) 
 \right\rangle,
\end{align}
where
\begin{equation}
 \mathcal{O}_{K}(w) \coloneqq \frac{1}{\varsigma(w)}\sum_{n=1}\frac{1}{n!} \sum_{\substack{\ell_1, \dots, \ell_n = 1 \\ \sum_j \ell_j = K}}  \prod_{j=1}^n \varsigma(w\ell_j)\frac{\partial}{\partial p_{\ell_j}}.
\end{equation}
The proof is now completed: the VEV equals $\mathcal{H}^{\wr,k, r+1}$ by definition, whereas the operator acting on the left equals $[w^{r+1}]. \E_{-k}(w)$ by the Boson-Fermion correspondence \eqref{eq:bosonfermion}.
\end{proof}
\end{theorem}

\vspace{1cm}
\section{Conclusion and open problems for future research}
\label{sec:conclusions}

We conclude the paper with some observations and propose some research questions related to neighboring fields of mathematics and physics for the future development of leaky Hurwitz numbers.

\begin{itemize}
\item[(Q1)] Several types of Hurwitz numbers are known to satisfy ELSV-type formulae, that is, they can be computed as integrals over moduli spaces of curves of certain cohomology classes. It would be interesting to determine an ELSV formulae for single connected completed cycles leaky Hurwitz numbers. In other words, establish the leaky equivalent of Zvonkine conjecture \cite{Z} (now theorem, see \cite{SolZ} and references therein). This is particularly interesting because $k$-leaky Hurwitz numbers with descendants have been initially \emph{defined} in terms of an ELSV-type formula \cite{CMS}, by intersecting the logarithmic $k$-double ramification cycle with the branch cycle in the logarithmic Chow ring found in \cite{CMR} and monomials of psi classes.
\item[(Q2)] By the work of Okounkov and Pandharipande, completed cycles are intrinsically related to Gromov-Witten theory via the GW/H correspondence \cite{OP}. Develop if possible a leaky Gromov-Witten theory corresponding to the leaky Hurwitz theory.
\item[(Q3)] Can one establish Topological recursion in the sense of Eynard and Orantin \cite{EO} for any variation of leaky Hurwitz numbers? If so, exhibit a spectral curve, possibly a quantum curve as well, and understand the necessary possible variations to the recursive procedure to produce the numbers. It would be interesting to understand what is the effect of having non-zero energy operators in the middle of the VEV on the Topological recursion procedure.
\item[(Q4)] Expand on the integrability satisfied by the partition function of Leaky Hurwitz numbers and describe the differences with the non-leaky case.
\item[(Q5)] Monotonicity (both strict and weak) in Hurwitz theory has proved over the last few years to provide fruitful results related to chamber polynomiality, wall crossing formulae, cut and join equations, topological recursion, integrability and ELSV formulae \cite{ALS, AKL, DK15}. It would be interesting to study if possible monotone leaky Hurwitz numbers.
\item[(Q6)] Same as in the previous question but for a leaky version of $b$ Hurwitz numbers in the sense of \cite{Hur-b} and references within.
\end{itemize}

\end{document}